\documentclass[preprint,review,10pt]{amsart}
\usepackage{amsmath,amssymb,amsfonts,xspace}
\newtheorem{theorem}{Theorem}[section]

\theoremstyle{definition}
\newtheorem{definition}[theorem]{Definition}
\newtheorem{example}[theorem]{Example}

\newtheorem{corollary}[theorem]{Corollary}
\newtheorem{lem}[theorem]{Lemma}

\theoremstyle{remark}

\numberwithin{equation}{section}

\begin{document}

\title{  (weakly) square-difference factor absorbing   hyperideals }

\author{Mahdi Anbarloei}
\address{Department of Mathematics, Faculty of Sciences,
Imam Khomeini International University, Qazvin, Iran.
}

\email{m.anbarloei@sci.ikiu.ac.ir}


\subjclass[2020]{  20N20, 16Y20  }


\keywords{  sdf-absorbing hyperideals, weakly sdf-absorbing hyperideals.}

\begin{abstract}
 In this paper, we introduce (weakly) square-difference factor absorbing   hyperideals in a multiplicative hyperring. 

\end{abstract}
\maketitle

\section{Introduction}

Marty's pioneering work on hyperstructures generalized many concepts in modern algebra by extending their underlying structures to hyperstructures \cite{marty}. This new algebraic framework has since been extensively studied and developed by numerous scholars \cite{f1,f2,f3,f4,f5,f7,f8,f9}.  A hyperoperation $``\circ" $ on non-empty set $Y$ is a mapping from $Y \times Y$ into $P^*(Y)$ such that $P^*(Y)$ is the family of all non-empty subsets of $Y$. In this case,  $(Y,\circ)$ is called hypergroupoid.  Let $Y_1,Y_2$ be two subsets of $Y$ and $y \in Y$, then $Y_1 \circ X_2 =\cup_{y_1 \in Y_1, y_2 \in Y_2}y_1 \circ y_2,$ and $ Y_1 \circ y=Y_1 \circ \{y\}.$ This means that the hyperoperation $``\circ" $ on $Y$ can be extended to  subsets of $Y$. A hypergroupoid $(Y, \circ)$ is called  a semihypergroup if $\cup_{a \in y \circ z}x \circ a=\cup_{b \in x \circ y} b \circ z $ for all $x,y,z \in Y$ which means $\circ$ is associative. A semihypergroup $Y$ is called a hypergroup if  $y \circ Y=Y=Y\circ y$ for each  $y \in Y$ \cite{f10}.  Multiplicative hyperrings as an important class of algebraic hyperstructures that generalize classical rings  were introduced by Rota in 1982 \cite{f14}. This class of hyperstructurs  has been widely reviewed in \cite{ameri5, ameri6,  anb4,   Kamali, Ghiasvand, Ghiasvand2, f16, ul}. Recall from \cite{f10} that an algebraic structure $(H,+,\circ)$ is a  commutative multiplicative hyperring if the following
holds: {\bf 1.} $(H,+)$ is an abelian  group,  {\bf 2.} $(H,\circ)$ is a semihypergroup; 
{\bf 3.} $x\circ (y+z) \subseteq x\circ y+x\circ z$ and $(y+z)\circ x \subseteq y\circ x+z\circ x$ for each $x, y, z \in H, $
{\bf 4.}  $x\circ (-y) = -(x\circ y)=(-x)\circ y$ for each $x, y \in H$,
{\bf 5.} $x \circ y =y \circ x$ for each $x,y \in H$.  The multiplicative hyperring $H$ is called strongly distributive if in (3)  the equality holds. 
Assume that $(\mathbb{Z},+,\cdot)$ is the ring of integers. For each subset $\Omega \in P^\star(\mathbb{Z})$ with $\vert \Omega\vert \geq 2$, there exists a multiplicative hyperring $(\mathbb{Z}_\Omega,+,\circ)$ where $\mathbb{Z}_\Omega=\mathbb{Z}$ and  $x \circ y =\{x.\gamma.y\ \vert \ \gamma \in \Omega\}$ for every $x,y\in \mathbb{Z}_\Omega$ \cite{das}.
An element $e \in H$ is considered as an identity element if $x \in x\circ e$ for every $x \in H$. An element $x \in H$ is called unit, if there exists $y \in H$ such that $e \in x \circ y$.  Denote the set of all unit elements in $H$ by $U(H)$  \cite{ameri}. Also, a multiplicative hyperring $H$ refers to a hyperfield if each non-zero element of $H$ is unit. A non-empty subset $A$ of  $H$ is a  hyperideal  if {\bf i.}  $x - y \in A$ for every $x, y \in A$,  {\bf ii.} $r \circ x \subseteq A$ for every $x \in A$ and $r \in H$ \cite{f10}.
 A proper hyperideal $A$ in  $H$ is a prime hyperideal if $x \circ y \subseteq A$ for $x,y \in H$ implies either  $x \in A$ or $y \in A$ \cite{das}. For any given hyperideal $A$ of $H$, $rad(A)$ denotes the intersection of all prime hyperideals of $H$ containing  $A$. If the multiplicative hyperring $H$  has no prime hyperideal containing $A$, we assume $rad(A)=H$. Suppose that $\mathcal{C}$ is the class of all finite products of elements of $H$ that is $\mathcal{C} = \{c_1 \circ c_2 \circ \cdots \circ c_n \ \vert \ c_i \in H, n \in \mathbb{N}\} \subseteq P^{\ast }(H)$ and  $A$ is a hyperideal of $H$. $A$ refers to a $\mathcal{C}$-hyperideal of $H$ if  $ A \cap C \neq \varnothing $ for all $C \in \mathcal{C}$ implies $C \subseteq A$.  Suppose that the hyperideal 
 $\{x \in H \ \vert \  x^n \subseteq A \ \text{for some} \ n \in \mathbb{N}\} $ is designated by $D(A)$. Note that the inclusion $D(A) \subseteq rad(A)$ always holds but other inclusion holds if $A$ is a $\mathcal{C}$-hyperideal of $H$ by Proposition 3.2 in \cite{das}. On the other hand, a hyperideal $A$ of $H$ is  a strong $\mathcal{C}$-hyperideal if $D \cap A \neq \varnothing$ for each $D \in \mathfrak{C}$  implies  $D \subseteq A$ such that $\mathfrak{C}=\{\sum_{i=1}^n C_i \ \vert \ C_i \in \mathcal{C}, n \in \mathbb{N}\}$ and $\mathcal{C} = \{c_1 \circ c_2 \circ \cdots \circ c_n \ \vert \ c_i \in H, n \in \mathbb{N}\}$ (for more details see \cite{phd}). 
  A proper hyperideal $A$ of  $H$ is maximal in $H$ if $A \subset B \subseteq H$ for
each hyperideal $B$ of $H$ implies  $B = A$ \cite{ameri}. The intersection of all maximal hyperideals of $H$ is denoted by $J(H)$. Also, $H$ refers to a local multiplicative hyperring if it has just one maximal hyperideal.   Assume that  $A_1$ and $A_2$ are hyperideals of $H$. We define $(A_2:A_1)=\{x \in H \ \vert \ x \circ A_1 \subseteq A_2\}$ \cite{ameri}.
A hyperring $H$ said to be of characteristic $\alpha$, if $\alpha$ is the smallest positive integer such that $\alpha x=0$ for all $x \in H$. If no such of $\alpha$  exists, then we say that $H$ is of characteristic $0$ \cite{Kamali}. The element $x \in H$ is nilpotent, if there exists an $n \in \mathbb{N}$ such that $0 \in x^n$. The set of all nilpotent elements of $H$ will be designated by $\Upsilon$ \cite{ameri}.

The concept of 2-absorbing hyperideals was studied in \cite{anb5}. In this paper, we aim to introduce and study the notions of  square-difference factor absorbing   hyperideals (or sdf-absorbing hyperideals) and weakly square-difference factor absorbing   hyperideals (or weakly sdf-absorbing hyperideals) in a commutative multiplicative hyperring.  Among many results in this paper, we conclude that for every  nonzero sdf-absorbing $\mathcal{C}$-hyperideal $P$ of $H$,  $rad(P)=P$ in Theorem \ref{1}. It is shown that  the converse of the explanation holds if the multiplicative hyperring $H$ is of characteristic 2 and  $P$ is a strong $\mathcal{C}$-hyperideal  in Theorem \ref{2}. We obtain that if every nonzero proper hyperideal of $H$ is an sdf-absorbing $\mathcal{C}$-hyperideal, then $H/\Upsilon$ is regular in Theorem \ref{7}. Although every nonzero prime strong $\mathcal{C}$-hyperideal is sdf-absorbing, Theorem \ref{12} shows the converse of the explanation holds if $1+1 \in U(H)$. We present some characterizations of these notions  on cartesian product of multiplicative hyperrings. 
Moreover, in Theorem \ref{01}, we show that if $P$ is a weakly sdf-absorbing strong $\mathcal{C}$-hyperideal of $H$ but is not an sdf-absorbing hyperideal, then $P \subseteq \Upsilon$.

Throughout this paper, $A$  denotes a commutative multiplicative hyperring with identity element $1$.

\section{  $sdf$-absorbing   hyperideals }
\begin{definition}
Let $P$ be a proper hyperideal of $H$. $P$ refers to a square-difference factor absorbing   hyperideal (or  sdf-absorbing hyperideal for short) if $0 \neq x, y \in H$ and $x^2 - y^2 \subseteq P$ imply $x-y  \in P$ or $x+y \in P$. 
\end{definition}
\begin{example}
Consider the multiplicative hyperring $(\mathbb{Z}_4,+,\circ)$, where the operation $+$ and the hyperoperation $\circ$ deﬁned by
\[ 
\begin{tabular}{|c|c|c|c|c|} 
\hline $+$ & $0$ & $1$ & $2$ & $3$
\\ \hline $0$ & $0$ & $1$ & $2$ & $3$
\\ \hline$1$ & $1$ & $2$ & $3$ & $0$
\\ \hline $2$ & $2$ & $3$ & $0$ & $1$
\\ \hline$3$ & $3$ & $0$ & $1$&$2$
\\ \hline
\end{tabular}\ \ \ \ \ \ \ \ \
\begin{tabular}{|c|c|c|c|c|} 
\hline $\circ$ & $0$ & $1$ & $2$ & $3$
\\ \hline $0$ & $\{0\}$ & $\{0\}$ & $\{0\}$ & $\{0\}$ 
\\ \hline $1$ & $\{0\}$ & $\mathbb{Z}_4$ & $\{0,2\}$ & $\mathbb{Z}_4$ 
\\ \hline$2$ & $\{0\}$ & $\{0,2\}$ & $\{0\}$ &$\{0,2\}$ 
\\ \hline $3$ & $\{0\}$ & $\mathbb{Z}_4$ & $\{0,2\}$ & $\mathbb{Z}_4$ 
\\ \hline
\end{tabular}\]
In this hyperring, $P=\{0,2\}$ is an sdf-absorbing hyperideal. 
\end{example}

\begin{theorem} \label{1}
If $P$ is a nonzero sdf-absorbing $\mathcal{C}$-hyperideal of $H$, then $rad(P)=P$.
\end{theorem}
\begin{proof}
Let $t \in rad(P)$. Then there exists $n \in \mathbb{N}$ such that $t^n \subseteq P$. Here we consider two cases. First, we assume that $n$ is an even number. So we have $t^{\frac{n}{2}} \circ t^{\frac{n}{2}} \subseteq P$. Take any $x \in t^{\frac{n}{2}}$. Therefore we have $x^2  \subseteq P$. We choose $0 \neq y \in P$ as  $P$ is a nonzero hyperideal of $H$. Hence we have $x^2 - y^2 \subseteq P$ which implies $x-y \in P$ or $x+y \in P$ as $P$ is an sdf-absorbing hyperideal of $H$. So we get $x \in P$. Since $P$ is a $\mathcal{C}$-hyperideal of $H$, we conclude that $t^{\frac{n}{2}}\subseteq P$. By continuing the process, we obtain $t \in P$. Now, we assume that $n$ is an odd number. From $t^n \subseteq P$, it follows that $t^{n+1} \subseteq P$. By a similar argument, we get $t \in P$ and so $rad(P) \subseteq P$. Since the inclusion $P \subseteq rad(P)$ always holds, we conclude that $rad(P)=P$. 
\end{proof}
\begin{theorem} \label{2}
Let $H$ be of  characteristic $2$ and $P$ be a proper strong $\mathcal{C}$-hyperideal of $H$. If $rad(P)=P$, then $P$ is an  sdf-absorbing hyperideal of $H$.
\end{theorem}
\begin{proof}
Assume that $x^2 - y^2 \subseteq P$ for $0 \neq x,y \in H$. Let $t \in x^2$, $s \in y^2$ and $r \in x \circ y$.  Since $H$ is of  characteristic $2$,  we have $t+s=t-s \in x^2 -y^2 \subseteq P$ and so $t+s+2r \in P$. Since $t+s+2r \in x^2+y^2+2 (x \circ y)$ and $P$ is a strong $\mathcal{C}$-hyperideal of $H$, we obtain $x^2+y^2+2 (x \circ y ) \subseteq P$ and so $(x+y)^2 \subseteq P$. Since $rad(P)=P$, we get $x+y \in P$. This shows that $P$ is an  sdf-absorbing hyperideal of $H$.
\end{proof}
\begin{theorem} \label{3}
Let $P$ be an  sdf-absorbing strong $\mathcal{C}$-hyperideal of $H$.  Then the followings are equivalent: 
\begin{itemize} 
\item[\rm(i)]~ $x^2-y^2 \subseteq P$ for $0 \neq x,y \in H$ implies $x-y \in P$ and $x+y \in P$.
\item[\rm(ii)]~ $1+1 \in P$.
\item[\rm(iii)]~ the hyperring $H/P$ is of characteristic $2$.
\end{itemize} 
\end{theorem}
\begin{proof}
(i)$  \Longrightarrow $(ii) Put $x=y=1$. Since $P$ is a strong $\mathcal{C}$-hyperideal of $H$ and $0 \in x^2-y^2$, we get $x^2-y^2 \subseteq P$. By the hypothesis, we conclude that $1+1=x+y \in P$, as needed.

(ii)$ \Longrightarrow$ (i) Let $x^2-y^2 \subseteq P$ for $0 \neq x,y \in H$. Since $P$ is an  sdf-absorbing hyperideal of $H$, we have $x-y \in P$ or $x+y \in P$. Since $P$ be a  strong $\mathcal{C}$-hyperideal of $H$, $(1+1) \circ x \subseteq (1 \circ x +1 \circ x) \cap P$ and $(1+1) \circ y \subseteq (1 \circ y +1 \circ y) \cap P$, we have $2x \in 1 \circ x+1 \circ x \subseteq P$ and $2y \in 1 \circ y+1 \circ y \subseteq P$ .  From $x-y \in P$, it follows that $x+y=x-y+2y \in P$.  Also, From $x+y \in P$, it follows that $x-y=x+y-2y \in P$.

(ii)$ \Longleftrightarrow$ (iii)  It is obvious.
\end{proof}
\begin{theorem} \label{4}
Assume that $P$ is a proper  strong $\mathcal{C}$-hyperideal of $H$. Then $P$ is an sdf-absorbing hyperideal of $H$ if and only if $x \circ y \subseteq P$ for $0 \neq x,y \in H \backslash P$ implies there are no $0 \neq a,b \in H$ that satisfy both equations  $A-B=x$ and $A+B=y$. 
\end{theorem}
\begin{proof}
$\Longrightarrow$ Assume that  $x \circ y \subseteq P$ for $0 \neq x,y \in H \backslash P$ and both equations  $A-B=x$ and $A+B=y$ has a solution in $H$ for some $0 \neq a,b \in H$. This implies that $(a-b) \circ (a+b) = x \circ y \subseteq P$. Since $P$ is a proper  strong $\mathcal{C}$-hyperideal of $H$ and $(a-b) \circ (a+b) \subseteq a^2+a \circ b -a \circ b -b^2$, we conclude that $a^2+a \circ b -a \circ b -b^2 \subseteq P$. Take any $t \in a^2$, $s \in b^2$ and $r \in a \circ b $. Since  $t+r-r-s \in a^2+a \circ b -a \circ b +b^2$ and $t-s \in a^2 -b^2$, we get $a^2 -b^2 \subseteq P$. Since $P$ is an sdf-absorbing hyperideal of $H$, we have $a-b=x \in P$ or $a +b=y \in P$ which is impossible.

$\Longleftarrow$ Let $a^2-b^2 \subseteq P$ for $0 \neq a,b \in H$. Assume that $a-b=x$ and $a+b=y$. Take any $t \in a^2$, $s \in b^2$ and $r \in a \circ b $. Since  $P$ is a proper  strong $\mathcal{C}$-hyperideal of $H$, $t-s \in a^2 -b^2$ and $t+r-r-s \in a^2+a \circ b -a \circ b +b^2$, we have $x \circ y =(a-b) \circ (a+b) \subseteq a^2 -a \circ b + a \circ b -b^2 \subseteq P$ and both equations  $A-B=x$ and $A+B=y$ has a solution in $H$ for $0 \neq a,b \in H$. Hence we have $a-b=x \in P$ or $a+b=y \in P$. Consequently,  $P$ is an sdf-absorbing hyperideal of $H$.
\end{proof}

An element $x \in H$ is said to be regular if there exists $y \in H$ such that $x \in x^2\circ y$. So, $H$ is regular if all of elements in $H$ are regular\cite{ameri5}. 
\begin{lem}\label{5}
Let every hyperideal of $H$ be a $\mathcal{C}$-hyperideal. Then $H$ is regular if and only if every hyperideal of $H$  is equal to its radical. 
\end{lem}
\begin{proof}
$\Longrightarrow$ Let $H$ be regular and $P$ be an arbitrary $\mathcal{C}$-hyperideal of $H$.  It is sufficient to show that $rad(P) \subseteq P$. Take any $x \in rad(P)$. Then we have $x^n \subseteq P$ for some $n \in \mathbb{N}$. Since $H$ is regular, there exists $y \in H$ such that $x \in x^2 \circ y$. Then we get $x \in x^2 \circ y \subseteq  x \circ x^2 \circ y^2 \subseteq x \circ x^ 3 \circ y^3 \subseteq  \cdots \subseteq x \circ x^n \circ y^n \subseteq P$. This shows that $rad(P) \subseteq P$. Since the inclusion $P \subseteq rad(P)$ always holds, we have $rad(P)=P$, as required. 

$\Longleftarrow$ Assume that every hyperideal  is equal to its radical. Then $x \in rad(\langle x^2 \rangle) =\langle x^2 \rangle$ for all $x \in H$. It follows that $x \in x^2 \circ y$ for some $y \in H$. Thus $x$ is regular for every $x \in H$ and so $H$ is regular.
\end{proof}
\begin{corollary}
Assume that $H$ is of characteristic $2$. If $H$ is regular, then every nonzero proper strong $\mathcal{C}$-hyperideal of $H$ is an sdf-absorbing hyperideal of $H$.
\end{corollary}
\begin{proof}
Let $H$ be of characteristic $2$. Since $H$ is regular, every hyperideal of $H$  is equal to its radical by Lemma \ref{5}. Therefore every nonzero proper strong $\mathcal{C}$-hyperideal of $H$ is an sdf-absorbing hyperideal of $H$ by Theorem \ref{2}. 
\end{proof}
Assume that $Q$ is a hyperideal of $(H,+,\circ)$. The the set $(H/Q=\{x+Q \ \vert \ x \in H\}, \oplus, *)$ is a multiplicative hyperring where $\oplus$ is the usual addition of cosets and $\star$ is defined by
\[(x+Q) *(y+Q)=\{z+Q \ \vert \ z \in x \circ y\} .\]
$H/Q$ is commutative if $H$ is so \cite{f10}.

\begin{theorem} \label{7}
 Let every nonzero proper hyperideal of $H$ be an sdf-absorbing $\mathcal{C}$-hyperideal of $H$. Then $H/\Upsilon$ is regular. Moreover,  there are no non-trivial chains of prime hyperideals of $H$.
\end{theorem}
\begin{proof}
Assume that every nonzero proper hyperideal of $H$ be an sdf-absorbing $\mathcal{C}$-hyperideal of $H$. Therefore we conclude that every hyperideal of $H$ is equal to its radicals by Theorem \ref{1}. This means that every hyperideal of $H/\Upsilon$ is equal to its radicals by Corollary 2.10 in \cite{Sen}. This implies that $H/\Upsilon$ is regular by Lemma \ref{5}. Now, let $P_1 \subsetneq P_2$ is a non-trivial chain of prime hyperideals of $H$ and $x \in P_2$. Since $H$ is regular, there exists $y \in H$ such that $x \in x^2 \circ y$. Since  $0 \in x^2 \circ y - x \cap P_1$, we have $x \circ (x \circ y - 1) \subseteq P_1$. This implies  $x \in P_1$  or $x \circ y -1 \subseteq P_1 \subseteq P_2$. The first case is impossible. In the second case, since $x \in P_2$,  we get  $1 \in P_2$ which is a contradiction. Thus  there are no non-trivial chains of prime hyperideals of $H$. 
\end{proof}
\begin{theorem}\label{8}
Let every nonzero proper hyperideal of $H$ be a $\mathcal{C}$-hyperideal and  $P$ be the only maximal hyperideal of $H$. Then every nonzero proper hyperideal of $H$ is an sdf-absorbing hyperideal of $H$ if and only if $P$ is the only prime hyperideal of $H$, $P$ is principal and $P^2=\{0\}$.
\end{theorem}
\begin{proof}
$\Longrightarrow$ Let every nonzero proper hyperideal of $H$ be an sdf-absorbing hyperideal of $H$. By Theorem \ref{7}, we conclude that $P$ is the only prime hyperideal of $H$. Assume that $I$ is an arbitrary nonzero hyperideal of $H$. By Theorem \ref{1}, we have $I=rad(I)=P$. This means  that $P$ is the only nonzero hyperideal of $H$ which means $P$ is principal and  $P^2=\{0\}$.

$\Longleftarrow$ Assume that $P$ is the only prime hyperideal of $H$, $P$ is principal and $P^2=\{0\}$. This implies that $P$ is the only nonzero proper hyperideal of $H$ and so  $P$ is an sdf-absorbing hyperideal of $H$.
\end{proof}
\begin{theorem}\label{9}
Let $H$ be a regular multiplicative hyperring and every hyperideal of $H$ be a $\mathcal{C}$-hyperideal of $H$. 
If there is only one maximal hyperideal $I$ of $H$ such that $H/I$ is not of characteristic $2$, then every  proper hyperideal of $H$ is an sdf-absorbing hyperideal of $H$.
\end{theorem}
\begin{proof}
Let $I$ be the only  maximal hyperideal of $H$ such that $H/I$ is not of characteristic $2$. Assume that $P$ is an arbitrary hyperideal of $H$. Since $H$ is regular, we conclude that $rad(P)=P$ by Lemma \ref{5}. Let us assume that $x^2 - y^2 \subseteq P$ for $x,y \in H$. If $P \subseteq I$, then we conclude that  $x-y \in I$ or $x+y \in I$. Now, let $J \neq I$ be a maximal hyperideal of $H$ with $P \subseteq J$. By the hypothesis, $H/J$ is of characteristic $2$. Since $x^2 - y^2 \subseteq J$, we get $x-y \in J$ and $x+y \in J$ by Theorem \ref{3}. This shows that  every prime (maximal) hyperideal of $H$ containing $I$ consists of   $x-y$ or every prime (maximal) hyperideal of $H$ containing $I$ consists of  $x+y$ which implies $I$ is an sdf-absorbing hyperideal of $H$.
\end{proof}
\begin{theorem}
Assume that $P_1,\cdots,P_n$ are prime strong $\mathcal{C}$-hyperideals of  $H$ such that the intersection of any $n-1$ of $P_i^,$s is not equal to $P_1 \cap \cdots \cap P_n$. Then $P_1 \cap \cdots \cap P_n$ is an sdf-absorbing  hyperideal of $H$ if and only if there is at most one $1 \leq i \leq n$ such that $H/P_i$ is not of characteristis $2$.
\end{theorem}
\begin{proof}
$\Longrightarrow$ Let $P=P_1 \cap \cdots \cap P_n$. Assume that $P$ is an sdf-absorbing  hyperideal of $H$ for $n \geq 2$ and  neither  $H/P_1$ nor $H/P_1$ are of characteristic $2$. Put $Q=P_2 \cap \cdots \cap P_{n}$. So $P \subsetneq Q$. Since $P \neq P_1$, there exist $x \in Q \backslash P_1$ and $y \in P_1 \backslash Q$. Assume that $a=x-y$ and $b=x+y$. . Therefore we have $a^2-b^2 \subseteq x^2 -2x \circ y  +y^2 -x^2-2x \circ y  -y^2 \subseteq P$. Since $P_1$ is a prime hyperideal of $H$ and $H/P_1$ is not of characteristic $2$, we conclude that  $x \circ (1+1) \nsubseteq  P_1$ by Theorem \ref{3} and so $x \circ 1 +x \circ 1 \nsubseteq P_1$. From $a+b=2x \in x \circ 1+ x \circ 1$, it follows that $a+b \notin P_1$ because $P_1$ is a strong $\mathcal{C}$-hyperideal of $H$ and $x \circ 1 +x \circ 1 \nsubseteq P_1$. Also, since $P_2$ is a prime hyperideal of $H$ and $H/P_2$ is not of characteristic $2$, we conclude that  $y \circ (1+1) \nsubseteq  Q$  and so $y \circ 1 +y \circ 1 \nsubseteq Q $. Since  $a-b=-2y \in -y \circ 1 -y \circ 1 \nsubseteq Q$ and $Q$ is a strong $\mathcal{C}$-hyperideal, we conclude that $a-b \notin Q$. Therefore neither $a-b$ nor $a+b$ are in $P$ which is a contradiction as $P$ is an sdf-absorbing  hyperideal of $H$.

$\Longleftarrow$ It is seen to be true in a similar manner to Theorem \ref{9}.
\end{proof}
Recall from \cite{f10} that a mapping $\theta$ from the commutative multiplicative hyperring 
$(H_1, +_1, \circ _1)$ into the commutative multiplicative hyperring $(H_2, +_2, \circ _2)$ is a hyperring good homomorphism if $\theta(x +_1 y) =\theta(x)+_2 \theta(y)$ and $\theta(x\circ_1y) = \theta(x)\circ_2 \theta(y)$ for all $x,y \in H_1$.
\begin{theorem} \label{homo}
Let $\theta$ be a good homomorphism from the commutative multiplicative hyperring $H_1$ into the commutative multiplicative hyperring $H_2$ and $P_1$ and $P_2$ be strong $\mathcal{C}$-hyperideals of $H_1$ and $H_2$, respectively. 
\begin{itemize} 
\item[\rm(i)]~ If $P_2$ is a nonzero sdf-absorbing hyperideal of $H_2$, then $\theta^{-1}(P_2)$ is  an sdf-absorbing hyperideal of $H_1$.
\item[\rm(ii)]~If $P_2$ is an sdf-absorbing hyperideal of $H_2$ and $\theta$ is injective, then $\theta^{-1}(P_2)$ is  an sdf-absorbing hyperideal of $H_1$.
\item[\rm(iii)]~ If $P_1$ is an sdf-absorbing hyperideal of $H_1$ and $\theta$ is surjective such that $ker (\theta) \subseteq P_1$, then $\theta(P_1)$ is  an sdf-absorbing hyperideal of $H_2$. Then the following statements are satisfied.
\end{itemize} 
\end{theorem}
\begin{proof}
(i) 
Let $x_1^2-x_2^2 \subseteq \theta^{-1}(P_2)$ for $x_1,x_2 \in H_1$. Take any $t_1 \in x_1^2$ and $t_2 \in x_2^2$. So $\theta(t_1-t_2) \in P_2$. Since  $P_2$ is a strong $\mathcal{C}$-hyperideals of $H_2$  and $\theta(t_1-t_2)=\theta(t_1)-\theta(t_2) \in \theta(x_1^2)-\theta(x_2^2)=\theta(x_1)^2-\theta(x_2)^2$, we get $\theta(x_1)^2-\theta(x_2)^2 \subseteq P_2$. Since $P_2$ is an sdf-absorbing hyperideal of $H_2$, we obtain $\theta(x_1-x_2)=\theta(x_1)-\theta(x_2) \in P_2$ or $\theta(x_1+x_2)=\theta(x_1)+\theta(x_2) \in P_2$. This implies that $x_1-x_2 \in \theta^{-1}(P_2)$ or $x_1+x_2 \in \theta^{-1}(P_2)$. Thus $\theta^{-1}(P_2)$ is  an sdf-absorbing hyperideal of $H_1$.

(ii) 
By using an argument similar to that in the proof of (i),
one can easily complete the proof.

(iii) Assume that $y_1^2-y_2^2 \subseteq \theta(P_1)$ for some $y_1, y_2 \in H_2$. Then there exist $x_1, x_2 \in H_1$ such that $\theta(x_1)=y_1$ and $\theta(x_2)=y_2$ as $\theta$ is surjective. So we have $\theta(x_1^2)-\theta(x_2^2) \subseteq \theta(P_1)$. Let $t_1 \in x_1^2$ and $t_2 \in x_2^2$. Then we have $\theta(t_1-t_2) \in \theta(P_1)$. This implies that $\theta(t_1-t_2)=\theta(t)$ for some $t \in P_1$. Therefore we have $\theta(t_1-t_2-t)=0$ which means $t_1-t_2-t \in Ker(\theta) \subseteq P_1$ and so $t_1 -t_2 \in P_1$. It follows $x_1^2-x_2^2 \subseteq P_1$ as $P_1$ is a strong $\mathcal{C}$-hyperideal. Since $P_1$ is an sdf-absorbing hyperideal of $H_1$, we conclude that $x_1-x_2 \in P_1$ or $x_1-x_2 \in P_1$. This means that $y_1-y_2=\theta(x_1-x_2) \in \theta(P_1)$ or $y_1+y_2=\theta(x_1+x_2) \in \theta(P_1)$. Hence $\theta(P_1)$ is  an sdf-absorbing hyperideal of $H_2$.
\end{proof}
Let us give the following theorem without proof as a result of Theorem \ref{homo}.
\begin{theorem} \label{10}
Let $P$ be a strong $\mathcal{C}$-hyperideal of $H$. Then the following statements are satisfied.
\begin{itemize} 
\item[\rm(i)]~ If $K$ is a multiplicative hyperring such that $K \subseteq H$ and  $P$ is an sdf-absorbing  hyperideal of $H$, then $P \cap K$ is an sdf-absorbing  hyperideal of $K$. 
\item[\rm(ii)]~ If $Q$ is a hyperideal of $H$ with $Q \subseteq P$ and $P$ is an sdf-absorbing  hyperideal of $H$, then  $P/Q$ is an sdf-absorbing  hyperideal of $H/Q$.
\item[\rm(iii)]~ Let $Q \subsetneq P$. Then $P$  is an sdf-absorbing  hyperideal of $H$  if and only if $P/Q$ is an sdf-absorbing  hyperideal of $H/Q$.
\end{itemize} 
\end{theorem}
Recall from \cite{ameri} that the hyperideals $P$ and $Q$ of $H$ are coprime if $P+Q=H$.
\begin{theorem} \label{11}
Assume that  $P_1,\cdots,P_n$ are coprime strong $\mathcal{C}$-hyperideals of  $H$. Then $P_1 \cap \cdots \cap P_n$ is an sdf-absorbing  hyperideal of $H$ if and only if there is at most one $1 \leq i \leq n$ such that $H/P_i$ is not of characteristis $2$.
\end{theorem}
\begin{proof}
$\Longrightarrow$ Put $P=P_1 \cap \cdots \cap P_n$. Suppose that $P$ is an sdf-absorbing  hyperideal of $H$ for $n \geq 2$ and  neither  $H/P_1$ nor $H/P_2$ are of characteristic $2$. By Theorem \ref{10} (ii), we conclude that $P/P$ is  an sdf-absorbing  hyperideal of $H/P$. By Corollary 3.17 in \cite{ameri}, we have $H/P \cong \Pi_{i=1}^n H/P_i$. Assume that $x=(1,1,\cdots,1)$ and $y=(-1,1,\cdots,1)$. Therefore we have $x^2-y^2 \subseteq \{(0,\cdots,0)\}$ and neither $x-y=(1+1,0,\cdots,0)$ nor $x+y=(0,1+1,\cdots,1+1)$ are in $\{(0,\cdots,0)\}$. This shows that $\{(0,\cdots,0)\}$ is not an sdf-absorbing  hyperideal of $\Pi_{i=1}^n H/P_i$ which is impossible. Consequently, there is at most one $1 \leq i \leq n$ such that $H/P_i$ is not of characteristis $2$.

$\Longleftarrow$ It can be easily seen that the claim is true in a similar manner to the proof of Theorem \ref{9}.

\end{proof}
Let  $H$ be a multiplicative hyperring. Then the set of all hypermatrices of $H$ is designated by $M_m(H)$. Assume that $A = (A_{ij})_{m \times m}$ and $ B = (B_{ij})_{m \times m}$ are in $ P^\star (M_m(H))$. Then $A \subseteq B$ if and only if $A_{ij} \subseteq B_{ij}$\cite{ameri}. 
\begin{theorem} 
Suppose that $P$ is a proper hyperideal of $H$. If $M_m(P)$ is an sdf-absorbing  hyperideal of $M_m(H)$, then $P$ is  an sdf-absorbing  hyperideal of $H$. 
\end{theorem}
\begin{proof}
Let $x^2- y^2 \subseteq P$ for  $0 \neq x,y \in H$. Then we get
\[\begin{pmatrix}
x^2- y^2 & 0 & \cdots & 0 \\
0 & 0 & \cdots & 0 \\
\vdots & \vdots & \ddots \vdots \\
0 & 0 & \cdots & 0 
\end{pmatrix}
\subseteq M_m(P).\]
Since $M_m(P)$ is an sdf-absorbing  hyperideal of  $M_m(H)$ and 
\[ \begin{pmatrix}
x^2- y^2 & 0 & \cdots & 0\\
0 & 0 & \cdots & 0\\
\vdots & \vdots & \ddots \vdots\\
0 & 0 & \cdots & 0
\end{pmatrix}
=
\begin{pmatrix}
x & 0 & \cdots & 0\\
0 & 0 & \cdots & 0\\
\vdots & \vdots & \ddots \vdots\\
0 & 0 & \cdots & 0
\end{pmatrix}^2
-
\begin{pmatrix}
y & 0 & \cdots & 0\\
0 & 0 & \cdots & 0\\
\vdots & \vdots & \ddots \vdots\\
0 & 0 & \cdots & 0
\end{pmatrix}^2
\]
we get the result that 
\[ \begin{pmatrix}
x & 0 & \cdots & 0\\
0 & 0 & \cdots & 0\\
\vdots& \vdots & \ddots \vdots\\
0 & 0 & \cdots & 0
\end{pmatrix} 
-
\begin{pmatrix}
y & 0 & \cdots & 0\\
0 & 0 & \cdots & 0\\
\vdots& \vdots & \ddots \vdots\\
0 & 0 & \cdots & 0
\end{pmatrix}=
\begin{pmatrix}
x-y& 0 & \cdots & 0\\
0 & 0 & \cdots & 0\\
\vdots& \vdots & \ddots \vdots\\
0 & 0 & \cdots & 0
\end{pmatrix}
\in M_m(P)\]
or 
\[ \begin{pmatrix}
x & 0 & \cdots & 0\\
0 & 0 & \cdots & 0\\
\vdots& \vdots & \ddots \vdots\\
0 & 0 & \cdots & 0
\end{pmatrix} 
+
\begin{pmatrix}
y & 0 & \cdots & 0\\
0 & 0 & \cdots & 0\\
\vdots& \vdots & \ddots \vdots\\
0 & 0 & \cdots & 0
\end{pmatrix}=
\begin{pmatrix}
x+y& 0 & \cdots & 0\\
0 & 0 & \cdots & 0\\
\vdots& \vdots & \ddots \vdots\\
0 & 0 & \cdots & 0
\end{pmatrix}
\in M_m(P).\]

This implies that $x-y \in  P$ or  $x+y \in  P$. Consequently, $P$ is an sdf-absorbing  hyperideal of $H$.
\end{proof}
\begin{theorem}\label{12}
Assume that $P$ is a nonzero sdf-absorbing strong $\mathcal{C}$-hyperideal of $H$ such that $1+1 \in U(H)$. Then $P$ is a prime hyperideal of $H$.
\end{theorem}
\begin{proof}
Suppose that $a\circ b \in P$ for $0 \neq a, b \in H$. Let us assume $b \neq a,-a$. Take any  $x \in a\circ (1+1)^{-1}+b\circ (1+1)^{-1}$ and $y \in a\circ (1+1)^{-1}-b\circ (1+1)^{-1}$. Therefore $x^2-y^2 \subseteq (a\circ (1+1)^{-1}+b\circ (1+1)^{-1})^2-(a\circ (1+1)^{-1}-b\circ (1+1)^{-1})^2 \subseteq a^2 \circ{(1+1)^{-1}}^2+a \circ b \circ{(1+1)^{-1}}^2+a \circ b \circ{(1+1)^{-1}}^2+b^2\circ{(1+1)^{-1}}^2-a^2 \circ{(1+1)^{-1}}^2+a \circ b \circ{(1+1)^{-1}}^2+a \circ b \circ{(1+1)^{-1}}^2-b^2\circ{(1+1)^{-1}}^2 \subseteq P$. Since $P$ is an sdf-absorbing hyperideal of $H$, we get $x-y \in P$ or $x+y \in P$. In the first case, we get $ b\circ (1+1)^{-1}+b\circ (1+1)^{-1} \subseteq P$ and so $b \in b \circ 1 \subseteq b \circ (1+1)^{-1} \circ (1+1) \subseteq b \circ (1+1)^{-1} \circ 1 +b \circ (1+1)^{-1} \circ 1 \subseteq P$. In the second case, we obtain  $ a\circ (1+1)^{-1}+a\circ (1+1)^{-1} \subseteq P$ which implies $a \in a \circ 1 \subseteq a \circ (1+1)^{-1} \circ (1+1) \subseteq a \circ (1+1)^{-1} \circ 1 +a \circ (1+1)^{-1} \circ 1 \subseteq P$. Now, let $b=a$ or $b=-a$. So we have $a^2 \subseteq P=rad(P)$ by Theorem \ref{1}. It follows that $a \in P$. Consequently, $P$ is a prime hyperideal of $H$.
\end{proof}
For given commutative multiplicative hyperrings $(H_1,+_1,\circ_1)$ and $(H_2,+_2,\circ_2)$   with nonzero identities, the triple  $(H_1 \times H_2,+,\circ) $ is a commutative multiplicative hyperring where

$(x_1,x_2)+(y_1,y_2)=(x_1+_1y_1,x_2+_2y_2)$

$(x_1,x_2) \circ (y_1,y_2)=\{(x,y) \in H_1 \times H_2 \ \vert \ x \in x_1 \circ_1 y_1, y \in x_2 \circ_2 y_2\}$ \\
for all  $x_1, y_1 \in H_1$ and $x_2, y_2 \in H_2$ \cite{ul}. Now, we give some characterizations of weakly sdf-absorbing hyperideals on cartesian product of commutative multiplicative hyperring.
\begin{theorem} \label{13}
Assume that $P_1$ and $P_2$ are nonzero proper strong $\mathcal{C}$-hyperideals of commutative multiplicative hyperrings $(H_1,+_1,\circ_1)$ and $(H_2,+_2,\circ_2)$, respectively. Then $P_1 \times P_2$ is an sdf-absorbing hyperideal of $H_1 \times H_2$ if and only if $P_1$ and $P_2$ are sdf-absorbing hyperideals of $H_1$ and $H_2$, respectively, and $1_{H_1}+1_{H_1} \in P_1$ or $1_{H_2}+1_{H_2} \in P_2$.
\end{theorem}
\begin{proof}
$\Longrightarrow$ Let $x_1^2-y_1^2 \subseteq P_1$ for $x_1,y_1 \in H_1$. Then we have $(x_1,0)^2-(y_1,0)^2 =(x_1^2-y_1^2,0^2-0^2) \subseteq P_1 \times P_2$. Since $P_1 \times P_2$ is an sdf-absorbing hyperideal of $H_1 \times H_2$, we get $(x_1-y_1,0)=(x_1,0)-(y_1,0) \in P_1 \times P_2$ or $(x_1+y_1,0)=(x_1,0)+(y_1,0) \in P_1 \times P_2$. Therefore we have $x_1-y_1 \in P_1$ or $x_1+y_1 \in P_1$. Thus $P_1$ is an sdf-absorbing hyperideal of $H_1$. Similarly, we can see that $P_2$ is an sdf-absorbing hyperideal of $H_2$. Now, put $x=(1_{H_1},1_{H_2}), y=(1_{H_1},-1_{H_2})$. Therefore we have $x^2-y^2 \subseteq P_1 \times P_2$. By the hypothesis, we obtain $(1_{H_1}+1_{H_1},0)=x+y \in P_1 \times P_2$ or $(0,1_{H_2}+1_{H_2})=x-y \in P_1 \times P_2$. This implies that $1_{H_1}+1_{H_1} \in P_1$ or $1_{H_2}+1_{H_2} \in P_2$.

$\Longleftarrow$ Assume that $(x_1,x_2)^2 -(y_1,y_2)^2 \subseteq P_1 \times P_2$ for $(0,0) \neq (x_1,x_2), (y_1,y_2) \in H_1 \times H_2$. Let us assume that $1_{H_1}+1_{H_1} \in P_1$. Therefore we get $x_1^2-y_1^2 \subseteq P_1$. Since $P_1$ is an  sdf-absorbing hyperideal of $H_1$ and $1_{H_1}+1_{H_1} \in P_1$, we conclude that $x_1-y_1 \in P_1$ and $x_1+y_1 \in P_1$ by Theorem \ref{3}. On the other hand, from $x_2^2-y_2^2 \subseteq P_2$, we have $x_2-y_2 \in P_2$ or $x_2+y_2 \in P_2$. In the first possibility, we get $(x_1,x_2)-(y_1,y_2)=(x_1-y_1,x_2-y_2) \in P_1 \times P_2$. In the second possibility, we obtain $(x_1,x_2)+(y_1,y_2)=(x_1+y_1,x_2+y_2) \in P_1 \times P_2$. Consequently, $P_1 \times P_2$ is an sdf-absorbing hyperideal of $H_1 \times H_2$.
\end{proof}

\begin{theorem} \label{14}
Assume that $P_1$ and $P_2$ are nonzero proper strong $\mathcal{C}$-hyperideals of commutative multiplicative hyperrings $(H_1,+_1,\circ_1)$ and $(H_2,+_2,\circ_2)$, respectively. Then $P_1$ is an sdf-absorbing hyperideal of $H_1$  if and only if  $P_1 \times H_2$ is an sdf-absorbing hyperideal of $H_1 \times H_2$.
\end{theorem}
\begin{proof}
It is proved in a similar way to Theorem \ref{13}.
\end{proof}
\begin{theorem} \label{15}
Let $(H_1,+_1,\circ_1)$ and $(H_2,+_2,\circ_2)$  be commutative multiplicative hyperrings such that $\{0\} \times H_2$ is a strong $\mathcal{C}$-hyperideal of $H_1 \times H_2$. Then  $\{0\}$ is an sdf-absorbing hyperideal of $H_1$ and the set of all nilpotent elements of $H_1$ is equal to $\{0\}$ if and only if $\{0\} \times H_2$ is an sdf-absorbing hyperideal of $H_1 \times H_2$. 
\end{theorem}
\begin{proof}
$ \Longrightarrow$ Let $\{0\}$ is an sdf-absorbing hyperideal of $H_1$ and the set of all nilpotent elements of $H_1$ be equal to $\{0\}$. Suppose that $(x_1,x_2)^2 -(y_1,y_2)^2 \subseteq \{0\} \times H_2$ for $(0,0) \neq (x_1,x_2), (y_1,y_2) \in H_1 \times H_2$.  So $x_1^2-y_1^2 \subseteq  \{0\}$. If  $0 \in x_1^2$ and $0 \in y_1^2$, then we get $x_1=0$ and $y_1=0$ as the set of all nilpotent elements of $H_1$ is equal to $\{0\}$, which is impossible. Since $\{0\}$ is an  sdf-absorbing hyperideal, we conclude that $x_1-y_1 =0$ or  $x_1+y_1 =0$. Hence $(0,x_2-y_2)=(x_1,x_2) -(y_1,y_2) \in \{0\} \times H_2$ or $(0,x_2+y_2)=(x_1,x_2) +(y_1,y_2) \in \{0\} \times H_2$. Thus  $\{0\} \times P_2$ is an sdf-absorbing hyperideal of $H_1 \times H_2$.

$\Longleftarrow$ Let $\{0\} \times H_2$ be an sdf-absorbing hyperideal of $H_1 \times H_2$. By Theorem \ref{14}, $\{0\}$ is  an sdf-absorbing hyperideal of $H_1$. Assume that $x \in H_1$ such that $0 \in x^n$ for some $n \in \mathbb{N}$. It is sufficient to show  $x=0$. Since $\{0\} \times H_2$ is a strong $\mathcal{C}$-hyperideal of $H_1 \times H_2$ and $(x^2-0^2,1_{H_1}^2-1_{H_2}^2) \cap \{0\} \times H_2$, we have $(x,1_{H_2})^2-(0,1_{H_2})^2 \subseteq \{0\} \times H_2$. This means that $(x,0)=(x,1_{H_2})-(0,1_{H_2}) \in \{0\} \times H_2$ or $(x,1_{H_2}+1_{H_2})=(x,1_{H_2})+(0,1_{H_2}) \in \{0\} \times H_2$. Hence $x=0$, as needed.
\end{proof}
\section{Weakly $sdf$-absorbing hyperideals}
\begin{definition}
Let $P$ be a proper hyperideal of $H$. $P$ is called  a weakly square-difference factor absorbing   hyperideal (or weakly  sdf-absorbing hyperideal for short) if $0 \neq x, y \in H$ and $0 \notin x^2 - y^2 \subseteq P$ imply $x-y  \in P$ or $x+y \in P$. 
\end{definition}
\begin{example}
Let $H=\{0,1,2,3\}$. Consider the multiplicative hyperring $(H+,\circ)$, where the operation $+$ and the hyperoperation $\circ$ deﬁned by
\[ 
\begin{tabular}{|c|c|c|c|c|} 
\hline $+$ & $0$ & $1$ & $2$ & $3$
\\ \hline $0$ & $0$ & $1$ & $2$ & $3$
\\ \hline$1$ & $1$ & $2$ & $3$ & $0$
\\ \hline $2$ & $2$ & $3$ & $0$ & $1$
\\ \hline$3$ & $3$ & $0$ & $1$&$2$
\\ \hline
\end{tabular}\ \ \ \ \ \ \ \ \
\begin{tabular}{|c|c|c|c|c|} 
\hline $\circ$ & $0$ & $1$ & $2$ & $3$
\\ \hline $0$ & $\{0\}$ & $\{0\}$ & $\{0\}$ & $\{0\}$ 
\\ \hline $1$ & $\{0\}$ & $\{1,3\}$ & $\{2\}$ & $\{1,3\}$ 
\\ \hline$2$ & $\{0\}$ & $\{2\}$ & $\{0\}$ &$\{2\}$ 
\\ \hline $3$ & $\{0\}$ & $\{1,3\}$ & $\{2\}$ & $\{1,3\}$ 
\\ \hline
\end{tabular}\]
In this hyperring, $P=\{0\}$ is a weakly  sdf-absorbing hyperideal. 
\end{example}
\begin{theorem} \label{01}
Let $P$ be a strong $\mathcal{C}$-hyperideal of $H$. If $P$ is a weakly sdf-absorbing hyperideal of $H$ but is not an sdf-absorbing hyperideal, then $P \subseteq \Upsilon$.
\end{theorem}
\begin{proof}
Assume that $P$ is not an sdf-absorbing hyperideal of $H$. Then we get  $0 \in x^2-y^2 \cap P$ for some $0 \neq x,y \in H$ such that neither $x-y$ nor $x+y$ are in $P$. If both $x$ and  $y $ are in $P$, we get $x-y,x+y \in P$ which is implossible. Let us assume that $y \notin P$. Suppose that $a \in P$. It is sufficient to show that $a$ is a nilpotent element of $H$. Since $y \notin P$, $y-a, y+a \neq 0$. From  $x^2-y^2 \subseteq P$, it follows that $x^2-(y+a)^2 \subseteq x^2-y^2-y \circ a -y \circ a -a^2 \subseteq P$ and $x^2-(y-a)^2 \subseteq x^2-y^2+y \circ a +y \circ a -a^2 \subseteq P$. If $0 \notin x^2-(y+a)^2$, then we obtain $x-y+a \in P$ or $x+y+a \in P$ as $P$ is a weakly sdf-absorbing hyperideal of $H$. This implies that $x-y \in P$ or $x+y \in P$ which is impossible.  Similarly, if $0 \notin x^2-(y-a)^2$, then we get a contradiction. Therefore we have $0 \in  x^2-(y+a)^2$ and $0 \in  x^2-(y-a)^2$. So $0 \in  x^2-(y+a)^2 \subseteq x^2-y^2-y \circ a -y \circ a -a^2$ and $0 \in  x^2-(y-a)^2 \subseteq x^2-y^2+y \circ a +y \circ a -a^2$. Then we conclude that $0 \in -y \circ a -y \circ a -a^2 +y \circ a +y \circ a -a^2$. Since $0 \in -y \circ a +y \circ a $, $0 \in 0+0-a^2$ and so $0 \in a^2$. This means $a \in \Upsilon$ which implies $P \subseteq \Upsilon$.
\end{proof}
We say that a proper hyperideal $P$ of $H$ is a weakly prime hyperideal if $0 \notin x \circ y \subseteq P$ for $x,y \in H$ imlpies $x \in P$ or $y \in P$.
\begin{theorem}\label{02}
Assume that $P$ is a nonzero sdf-absorbing strong $\mathcal{C}$-hyperideal of $H$ such that $1+1 \in U(H)$. Then $P$ is a prime hyperideal of $H$.
\end{theorem}
\begin{proof}
It can be easily seen that the claim is true in a similar manner to the proof of Theorem \ref{12}.
\end{proof}
\begin{theorem} \label{03}
Let $(H_1,+_1,\circ_1)$ and $(H_2,+_2,\circ_2)$  be commutative multiplicative hyperrings and $P_1$ be a nonzero weakly sdf-absorbing strong $\mathcal{C}$-hyperideal of $H_1$.  Then  the followings are equivalent. 
\begin{itemize} 
\item[\rm(i)]~$P_1$ is an sdf-absorbing hyperideal of $H_1$.
\item[\rm(ii)]~$P_1 \times H_2$ is an sdf-absorbing hyperideal of $H_1 \times H_2$.
\item[\rm(iii)]~$P_1 \times H_2$ is a weakly  sdf-absorbing hyperideal of $H_1 \times H_2$. 
\end{itemize} 
\end{theorem}
\begin{proof}
(i) $\Longrightarrow $ (ii) The claim follows from Theorem \ref{14}. 

(ii) $\Longrightarrow $ (iii) It is obvious.

(iii) $\Longrightarrow $ (i) Suppose that $P_1 \times H_2$ is a weakly  sdf-absorbing hyperideal of $H_1 \times H_2$ but $P_1$ is not an sdf-absorbing hyperideal of $H_1$. Then there exist $0 \neq x,y \in H_1$ with $0 \in x^2-y^2 \cap P_1$ but $x-y \notin P_1$ and $x+y \notin P_1$. Put $a=(x,1_{H_2})$ and $b=(y,0)$. So we have $(0,0) \notin a^2-b^2 \subseteq P_1 \times H_2$ and $0 \neq a,b \in H_1$. By the hypothesis, we obtain $(x-y,1_{H_2})=a-b \in P_1 \times H_2$ or $(x+y,1_{H_2})=a+b \in P_1 \times H_2$. This means that $x-y \in P_1$ or $x+y \in P_1$. This is impossible. Thus  $P_1$ is an sdf-absorbing hyperideal of $H_1$.
\end{proof}
\begin{theorem} \label{04}
Let   $P_1$ be a  weakly sdf-absorbing strong $\mathcal{C}$-hyperideal of $(H_1,+_1,\circ_1)$ but is not  an sdf-absorbing hyperideal and $P_2$ be a  weakly sdf-absorbing strong $\mathcal{C}$-hyperideal of $(H_2,+_2,\circ_2)$ but is not  an sdf-absorbing hyperideal.  Then  the followings are equivalent. 
\begin{itemize} 
\item[\rm(i)]~$P_1 \times P_2$ is a weakly  sdf-absorbing hyperideal of $H_1 \times H_2$ but is not an sdf-hyperideal hyperideal.
\item[\rm(ii)]~$P_1 \times P_2$ is a weakly  sdf-absorbing hyperideal of $H_1 \times H_2$.
\item[\rm(iii)]~ If $x^2-y^2 \subseteq P_1$ for $x,y \in H_1$, then $0 \in x^2-y^2$, and if $u^2-v^2 \subseteq P_2$ for $u,v \in H_2$, then $0 \in u^2-v^2$.
\item[\rm(iv)]~ If $a^2-b^2 \subseteq  P_1 \times P_2$ for $(0,0) \neq a,b \in H_1 \times H_2$, then $(0,0) \in a^2-b^2$.
\end{itemize} 
\end{theorem}
\begin{proof}
(i) $\Longrightarrow$ (ii) It is obvious.

(ii) $\Longrightarrow$ (iii) Suppose that $0 \notin x^2-y^2 \subseteq P_1$ for $x,y \in H_1$. Since $P_2$ is not  an sdf-absorbing hyperideal of $H_2$, we get $0 \in t^2-s^2 \cap P_2$ for some $0 \neq t,s \in H_1$ such that  neither $t-s \in P_2$ nor $t+s \in P_2$. Put $a=(x,t)$ and $b=(y,s)$. This implies that $(0,0)  \notin a^2-b^2=(x^2-y^2,t^2-s^2) \in P_1 \times P_2$ for $(0,0) \neq a,b \in H_1 \times H_2$. Since $P_1 \times P_2$ is a weakly  sdf-absorbing hyperideal of $H_1 \times H_2$, we conclude that $a-b=(x-y,t-s) \in P_1 \times P_2$ or $a+b=(x+y,t+s) \in P_1 \times P_2$. It follows that $t-s \in P_2$ or $t+s \in P_2$, which is impossible. Therefore  $0 \in x^2-y^2$. Now, let $u^2-v^2 \subseteq P_2$ for $u,v \in H_2$. Similarly, we can show that $0 \in u^2-v^2$.

(iii) $\Longrightarrow$ (iv) It is obvious. 

(iv) $\Longrightarrow$ (i)  It is clear that $P_1 \times P_2$ is a weakly  sdf-absorbing hyperideal of $H_1 \times H_2$. By the hypothesis, we get $0 \in x^2-y^2 \cap P_1$ for some $0 \neq x,y \in H_1$ such that neither $x-y \in P_1$ nor $x+y \in P_1$. Put $a=(x,0)$ and $b=(0,y)$. This implies that $(0,0) \in a^2-b^2 \cap  P_1 \times P_2$ such that neither $a-b=(x-y,0) \in P_1 \times P_2$ nor $a+b=(x+y,0) \in P_1 \times P_2$. Consequently, $P_1 \times P_2$ is not an sdf-hyperideal hyperideal.
\end{proof}


\end{document}